\documentclass[12pt,reqno,twoside]{amsart}
\usepackage[dvips,bottom=1.4in,right=1in,top=1in, left=1in]{geometry}
\usepackage[latin1]{inputenc}
\usepackage{amsfonts,amsmath,amssymb}
\usepackage{mathrsfs,mathtools}
\usepackage{algorithm}
\usepackage{algpseudocode}
\usepackage{enumerate}
\usepackage{hyperref}
\usepackage{esint}
\usepackage{graphicx,subfigure}
\usepackage{bm}
\usepackage{commath}
\usepackage{esint}
\usepackage{placeins}
\usepackage{flafter}
\usepackage{tikz}
\usepackage[
backend=biber,
style=alphabetic,
]{biblatex}
\usetikzlibrary{patterns}

\usepackage{lineno}
\usepackage{caption}
\usepackage{placeins}

\DeclareMathAlphabet{\mathpzc}{OT1}{pzc}{m}{it}
\numberwithin{equation}{section}

\makeatletter
\def\eqnarray{\stepcounter{equation}\let\@currentlabel=\theequation
	\global\@eqnswtrue
	\tabskip\@centering\let\\=\@eqncr
	$$\halign to \displaywidth\bgroup\hfil\global\@eqcnt\z@
	$\displaystyle\tabskip\z@{##}$&\global\@eqcnt\@ne
	\hfil$\displaystyle{{}##{}}$\hfil
	&\global\@eqcnt\tw@ $\displaystyle{##}$\hfil
	\tabskip\@centering&\llap{##}\tabskip\z@\cr}
\def\endeqnarray{\@@eqncr\egroup
	\global\advance\c@equation\m@ne$$\global\@ignoretrue}

\newtheorem{theorem}{Theorem}[section]

\newtheorem{definition}[theorem]{Def{}inition}

\numberwithin{equation}{section}

\usepackage[textsize=small]{todonotes}
\newcommand{\R}{\mathbb{R}}
\newcommand{\cU}{\mathcal{U}}
\newcommand{\cV}{\mathcal{V}}

\DeclareMathOperator{\supp}{supp}

\title{A Generic MATLAB Toolbox to Approximate PDEs Using Computational Geometry}
\date{\today}

\thanks{
	This work is partially supported by NSF grant DMS-2408877, Air Force Office of Scientific Research (AFOSR) under Award NO: FA9550-22-1-0248, and Office of Naval Research (ONR) under Award NO: N00014-24-1-2147.
}

\author{Kiefer Green and Harbir Antil}
\address{K. Green and H. Antil. The Center for Mathematics and Artificial Intelligence
	(CMAI) and Department of Mathematical Sciences, George Mason University,
	Fairfax, VA 22030, USA.}
\email{kgreen32@gmu.edu, hantil@gmu.edu}
\begin{document}	

%%%%%%%%%%%%%%%%%%%%%%%%%%%%%%%%%%%%%%%%%%%%%%%%%%%%%%
 
	\begin{abstract}	
		This article introduces a general purpose framework and software to approximate partial differential equations (PDEs). The sparsity patterns of finite element discretized operators is identified automatically using the tools from computational geometry. They may enable  experimentation with novel mesh generation techniques and could simplify the implementation of methods such as multigrid. We also implement quadrature methods following the work of Grundmann and M\"{o}ller. These methods have been overlooked in the past but are more efficient than traditional tensor product methods. The proposed framework is applied to several standard examples.
	\end{abstract}
	
	
%%%%%%%%%%%%%%To Be Commented Out %%%%%%%%%%%%%%%%%%%	
\maketitle
%\tableofcontents
%%%%%%%%%%%%%%%%%%%%%%%%%%%%%%%%%%%%%%%%%%%%%%%%%%%%%%

%%%%%%%%%%%%%%%%%%%%%%%%%%%%%%%%%%%%%%%%%%%%%%%%%%%%%%	
\FloatBarrier\section{Introduction} \label{sec:Intro}

Partial differential equations (PDEs) are needed to model many physical phenomena, and the fields of mathematics, physics and engineering encounter them regularly. Each PDE necessitates a tool to be built that can model the data and summarize the outcomes for different variables. The Finite Element Method is the go-to method for stationary or dynamic PDEs. 

Current systems are specialized and opaque. Historically, systems have been created to solve a specific PDE at hand. Their strength is that they are very good at solving the problem they are built for. The weakness is that a new system has to be created for every problem. Current systems also rely on assumptions and, therefore, lack generality. The PDE Library aims to avoid those assumptions (to some extent) and to allow solutions to be found for more general problems. 

Conventional tools for solving PDEs are currently hyper-specialized and require expert knowledge to create. Conversely, the PDE Library would create a more generalized tool kit that has broader applications and would eliminate the initial need to build a specialized tool to solve any given PDE. Furthermore, the PDE Library would ease the description of variational problems so that different models can be explored for applied problems. A successful PDE Library would need to adhere to three ideals: Generality, Flexibility, and Transparency. 

\FloatBarrier\subsection*{Ideals:}
Current systems are specialized and opaque. We will address this by developing our toolkit with the following ideals in mind.

\textbf{Generality:} Usable for a broad array of problems.
%Currently, PDEs require software that is specifically designed to solve for that PDE. 
{Existing PDE solvers are largely specialized.  
For example, an engineer could build a tool that could analyze stresses endured by a bridge, but they may need a different tool to model the heat produced by the materials used to build that same bridge. By creating a PDE Library that adheres to the Generality Ideal, we provide one tool kit that could be used to model both heat and stress.

\textbf{Flexibility:} Able to mesh easily with novel codes and ideas.
Current PDE tools have the benefit of being incredibly powerful because they are typically tailored to the specific PDE they were built for, but take time to create and refine. At times, this makes them inflexible. The PDE Library adheres to the Flexibility Ideal and allows the user to use different tools for each part of the process. Another benefit of a flexible toolkit is the ability to run in multiple workflows. An effective tool will not lock out users based on their systems. The PDE Library adheres to the Portability Ideal and is able to be taken somewhere else and be used just as easily and effectively in a wider range of scenarios.

\textbf{Transparency:} Academic code is a record of knowledge, and the knowledge recorded in this code should be transparent. An interested user with an appropriate background should be able to look at the code and understand how it works. Current PDE software is usable but it can be challenging to understand what is under the hood. Proprietary software bars access to the code, and open source software is not always concerned with the record of knowledge. 
%A better open source tool should not just be available, it must also be sensible. 
The PDE Library adheres to the Transparency Ideal and maintains a more accurate academic record.

The purpose of this tool is to make it quicker to produce solvers for novel PDEs. The PDE Library has broad applications in academia which is inherently experimental and benefits from tools that have broad uses that can be applied to a wider range of problems. Additionally, the PDE library can expand the capabilities of a non-specialized user who is not an expert in the subject. 

\FloatBarrier\subsection*{Prior Art and existing PDE Solvers}
Current PDE solvers and structural analysis software used to solve PDEs have strengths and weaknesses that the PDE Library aims to address. While some have low barriers to entry, they are not adaptable, others are more general in their use, but require subject expertise. The PDE Library focused on the characteristics of three existing solvers (Autodesk \cite{reinhart2009experimental}, MATLAB \cite{MR1787308}, and FEniCS \cite{alnaes2015fenics}) when evaluating which characteristics to apply to an academically open source tool that maintains the academic record.

\textbf{Autodesk:} The Autodesk model is one tool for one PDE. This makes it incredibly well-suited for solving that specific PDE, but fails to adhere to the Generality Ideal. Furthermore, by design it doesn't solve PDEs, rather it analyzes structures. Autodesk is also proprietary code and fails the Transparency Ideal. 

\textbf{MATLAB:} MATLAB is a better PDE solver, and it is usable by anyone who has at least an undergraduate level understanding of PDEs, but it is limited because it can only solve one specific type of PDE making it fail the Generality and Flexibility Ideals. It is more general than Autodesk, but can still be improved upon.

\textbf{FEniCS:} Fenics is an open source general PDE solver that adheres to the Ideal of Generality. It can solve any PDE the user needs so long as they are sufficiently knowledgeable about the PDE they need to solve. Fenics is built around the Unified Form Language (UFL) which requires the user to describe the problem at which point it is passed to the backend. However, this obfuscates the code and doesn't maintain an academic record, making it fail the Ideal of Transparency. Furthermore, Fenics requires UFL from start to finish, making it fail the Flexibility Ideal. 
%Finally, one doesn't have to be an expert to use Fenics, but they have to be an expert to understand the data. 

The PDE Library aims to overcome some of the issues with these prior systems by being general and open source, maintaining an academic record of how it works, and by keeping a simple data transfer model that makes it possible to use external tools for individual parts of the PDE solve. The remainder of the paper will be organized as follows: We state some preliminary notations in Section \ref{sec2}. Section \ref{sec3} deals with the description of how the tool handles the Finite Element Method and demonstrates the Flexibility Ideal. Later, in Section \ref{sec4}, we describe a novel way of handling operator sparsity which enables experimentation with novel mesh generation techniques and simplifies the implementation of methods such as multigrid. When implementing multigrid two meshes are present, the course and fine meshes. Methods that rely on determining sparsity during mesh generation cannot be used for the interpolation and prolongation operators unless the coarse grid is embedded in the fine grid because interaction of elements is no longer determined by adjacency in the mesh. Our method determines sparsity from the basis functions and allows us to handle these operators efficiently. This also helps us follow the flexibility ideal. \\
In Section \ref{sec5}, we describe how quadrature is implemented following Grundmann and M\"{o}ller \cite{GrundmannMoller} to avoid the cost of tensor product quadrature methods  in high dimension. These methods have been overlooked in the past, but with changes in floating-point arithmetic, they are now practical to use for our needs which we will show by analyzing the quadrature error of these methods. By being open about what quadrature methods we are using we uphold the Transparency Ideal, maintaining the academic record. In Section \ref{sec6}, we mention the numerical scheme and present a slate of examples that show the broad applicability of the tool and the Generality Ideal.

%%%%%%%%%%%%%%%%%%%%%%%%%%%%%%%%%%%%%%%%%%%%%%%%%%%%%%

\FloatBarrier\section{General Remarks}\label{sec2}
Let $\cU$ and $\cV$ be linear spaces. Consider the following general problem:  
Find $u\in\cU$ such that 
$$a(u,v)=\ell(v) \quad \forall~ v \in \cV \, .$$
The corresponding discrete problem reads as: Find $u_h\in\cU_h$ such that 
$$a_h(u_h,v_h)=\ell_h(v_h) \quad \forall~ v_h \in\cV_h,$$
where $a:\cU\times\cV\to\R^n$ and $a_h:\cU_h\times\cV_h\to\R^n$ are linear in their second arguments and $\ell:\cV\to\R^n$ and $\ell_h:\cV_h\to\R^n$ denote the continuous linear functionals. In particular, many PDEs in their weak-form lead to problems of the form: find $u\in\cU$ such that 
$$\sum_i\int_\Omega f_i(u,x)g_i(v,x)=\ell(v)$$
for all $v\in\cV$.

\begin{definition}[Support Preserving]
\label{def:supp_pres}
We call a function $f:\cU\times\Omega\to\R$ \textbf{support preserving} if $f(u,x)$ has the same support as $u(x)$ for any $u\in\cU$.
\end{definition}
We often find that local PDEs lead to variational problems of the form: find $u\in\cU$ such that 
$$\sum_i\int_\Omega f_i(u,x)g_i(v,x)=\ell(v)$$
for all $v\in\cV$ where all of the $f_i$ and $g_i$ are support preserving. Lastly, given an $n$-dimensional simplex $S$, which is the convex hull of the points $\{x_1,x_2,\dots,x_n\}$, and a function $f$ defined on $S$, we let 
\begin{equation}\label{permsum}
f((c_1,c_2,\dots,c_n,c_{n+1})):=\sum_{(a_1,\dots,a_{n+1})\in A(c_1,\dots,c_{n+1})}f(a_1 x_1+a_2 x_2+\dots+a_n x_n)
\end{equation}
where $A(c_1,\dots,c_{n+1})$ denotes the set of all distinct permutations of the tuple $(c_1,\dots,c_{n+1})$.
We will find this notation useful when defining quadrature methods. 

\FloatBarrier\section{Finite Element Spaces}\label{sec3}
Finite element method (FEM) will be used to discretize PDEs. We introduce a finite element and refer to \cite{BrennerScott} for details.
\begin{definition} Let
\begin{itemize}
    \item $K\subseteq\R^n $ be a bounded closed set with nonempty interior and piecewise smooth boundary (the \textbf{element domain})
    \item $\mathcal{P} $ be a finite-dimensional space of functions on $K$ (the space of \textbf{shape functions}) and
    \item $\mathcal{N}=\{N_1,N_2,\cdots,N_k\} $ be a basis for $\mathcal{P}' $, the dual space of $\mathcal{P}$ (the set of \textbf{nodal variables}).
\end{itemize}
Then $(K,\mathcal{P},\mathcal{N})$ is called a \textbf{finite element}.
\end{definition}
For computational purposes, we will work with finite element basis functions defined as an object with field $\mathtt{Support:List\langle Simplex\rangle}$ which is the set of simplices which make up the support of the function and method
$$\mathtt{Evaluate:List\langle Point\rangle\times List\langle Vector \rangle\to List\langle Derivatives\rangle}$$ 
which computes the requested derivatives of the basis function at the given points. At this point we notice that we are restricting ourselves to triangular meshes. This is reasonably general as any polytope can be triangulated, for example see Figure~\ref{fig:Floor}. Doing things this way fulfills the ideal of flexibility by simplifying the mode of passing mesh and basis information around and allows for more experimentation with new methods.

\begin{figure}[hbt!]
\caption{A conforming triangulation of the floor plan of the fourth floor of Exploratory Hall at George Mason University, Fairfax, Virginia.}
\centering
\includegraphics[width=0.8\textwidth]{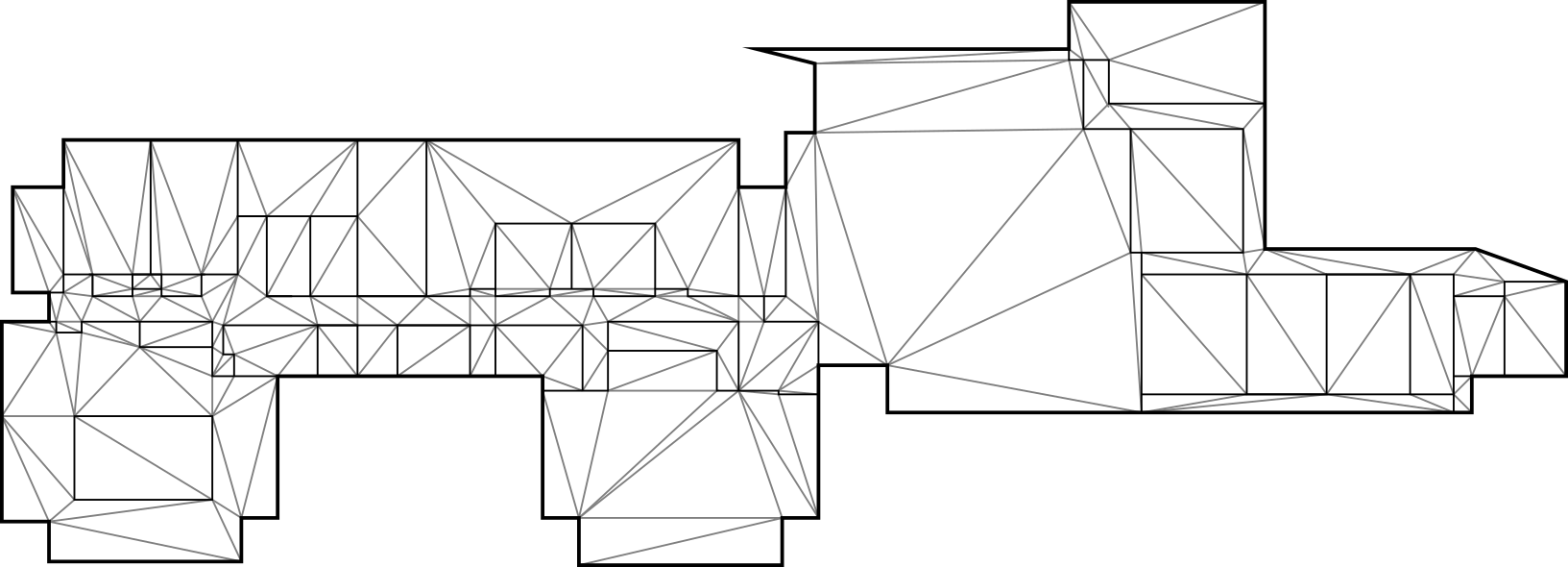}
\label{fig:Floor}
\end{figure}

\FloatBarrier\section{Geometry}\label{sec4}
We need to be able to handle an operator of type
$$a(u,v)=\int_\Omega f(u,x)g(v,x)dx$$
but if $f$ and $g$ are support preserving (see Definition~\ref{def:supp_pres}), as they frequently are, we can apply the following theorem.
\begin{theorem}
If $a(u,v)=\int_\Omega f(u,x)g(v,x)dx$ where $f$ and $g$ are Support-Preserving. Then,$$a(u,v)=\int_{S_{u}\cap S_{v}} f(u,x)g(v,x)dx$$
where $\supp{u}\subset S_u$ and $\supp{v}\subset S_v$.
\end{theorem}
\begin{proof}
The proof follows from a simple argument:
\begin{align*}
    a(u,v)&=\int_\Omega f(u,x)g(v,x)dx\\
    &=\int_{\Omega\setminus(S_{u}\cap S_{v})} f(u,x)g(v,x)dx+\int_{S_{u}\cap S_{v}} f(u,x)g(v,x)dx\\
    &=\int_{\Omega\setminus(S_{u}\cap S_{v})} 0 dx+\int_{S_{u}\cap S_{v}} f(u,x)g(v,x)dx\\
    &=\int_{S_{u}\cap S_{v}} f(u,x)g(v,x)dx
\end{align*}
which completes the proof.
\end{proof}

In particular, we can take $S_u$ and $S_v$ as the convex hulls of the supports of $u$ and $v$ which allows us to consider the intersection of convex polytopes in order to determine pairs $u$ and $v$  such that $a(u,v)=0$.

%In the following picture we see three cases of convex figures intersecting. On the left are two that positively intersect, and so have to be handled by the system in the middle are two shapes that do not intersect and we can geometrically determine that they do not have to be considered. And, finally, on the right are two that geometrically intersect but do not contribute to the finite element solution. 

%\todo{What is the purpose of this figure, explain this in the text and add reference to the figure}
%\begin{figure}[hbt!]
%\caption{Three manners of intersection for convex sets. The first is a collision that cannot be ignored. The second is what the collision test allows us to ignore. The third can be ignored but is not captured by the collision test.}
%\centering
%\includegraphics[width=\textwidth]{Intersections.png}
%\label{fig:Intersect}
%\end{figure}

The problem of determining whether convex figures intersect is well handled in computational geometry and our approach is based on the separating axis theorem. We will provide the proof of this because it clarifies how we will use it. 

\begin{theorem}
    Given any two compact and convex sets $U$ and $V$, which are subsets of $\R^n$, then $U$ and $V$ are disjoint if and only if there is a vector $\Vec{v}$ called a separating axis such that the scalar projections $\pi_{\Vec{v}}(U)$ and  $\pi_{\Vec{v}}(V)$ are disjoint intervals. 
\end{theorem}
\begin{proof}
First direction: Suppose that there is a vector $\Vec{v}$ such that the scalar projections $\pi_{\Vec{v}}(U)=[U_m,U_M]$ and $\pi_{\Vec{v}}(V)=[V_m,V_M]$ are disjoint intervals and suppose that $x\in U\cap V$. Now, $\pi_{\Vec{v}}(x)\in [U_m,U_M]\cap [V_m,V_M]$, which is a contradiction as the intervals are disjoint, therefore $U\cap V=\varnothing$. The left panel of figure \ref{fig:Separate} shows what is meant by a separating axis.

%\todo{Add reference to this figure in the text and describe what it says!}
%\begin{figure}[hbt!]
%\caption{Sketch of first direction}
%\centering
%\includegraphics[width=.5\textwidth]{Sep1.png}
%\end{figure}

Second direction: Suppose that $U$ and $V$ are disjoint and let $\Vec{v}=p_U-p_V$ where $(p_U,p_V)\in U\times V$ solves the minimization problem
\begin{align*}
    \min_{(p_U,p_V)\in U\times V}\|p_U-p_V\|^2 \, .
\end{align*}
This minimum exists, since $U\times V$ is compact, and is positive, since $U$ and $V$ are disjoint. Now, suppose $\pi_{\Vec{v}}(U)\cap\pi_{\Vec{v}}(V)\neq\varnothing$, that is there is a point $a\in U$ such that $\pi_{\Vec{v}}(a)\in\pi_{\Vec{v}}(V)$. Now, find on the segment connecting $p_U$ to $a$ the point $a'$ which is closest to $p_V$. This point must be in $U$ since $U$ is convex but this is a contradiction since $\|a'-p_V\|^2<\|p_U-p_V\|^2$. Thus the projections are disjoint. The right panel of figure \ref{fig:Separate} shows that the axis connecting closest points separates.
 	\begin{figure}[hbt!]		
		\centering
		\includegraphics[width=0.45\textwidth]{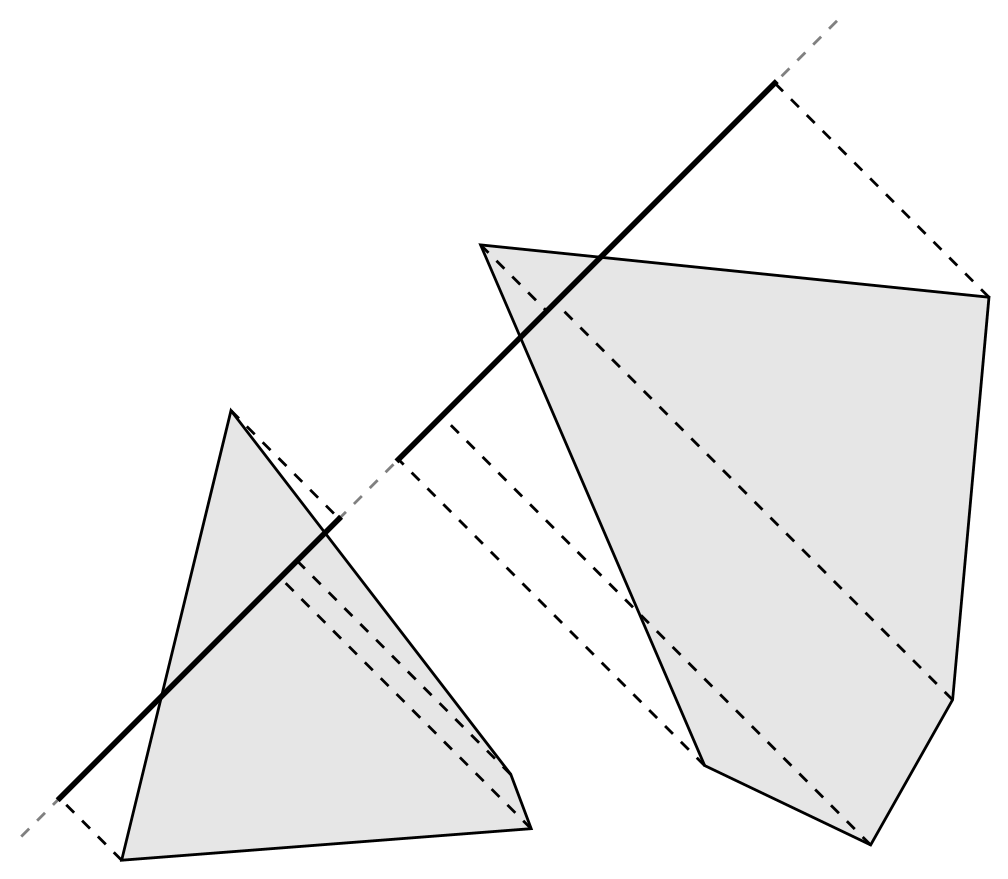}
        \includegraphics[width=0.45\textwidth]{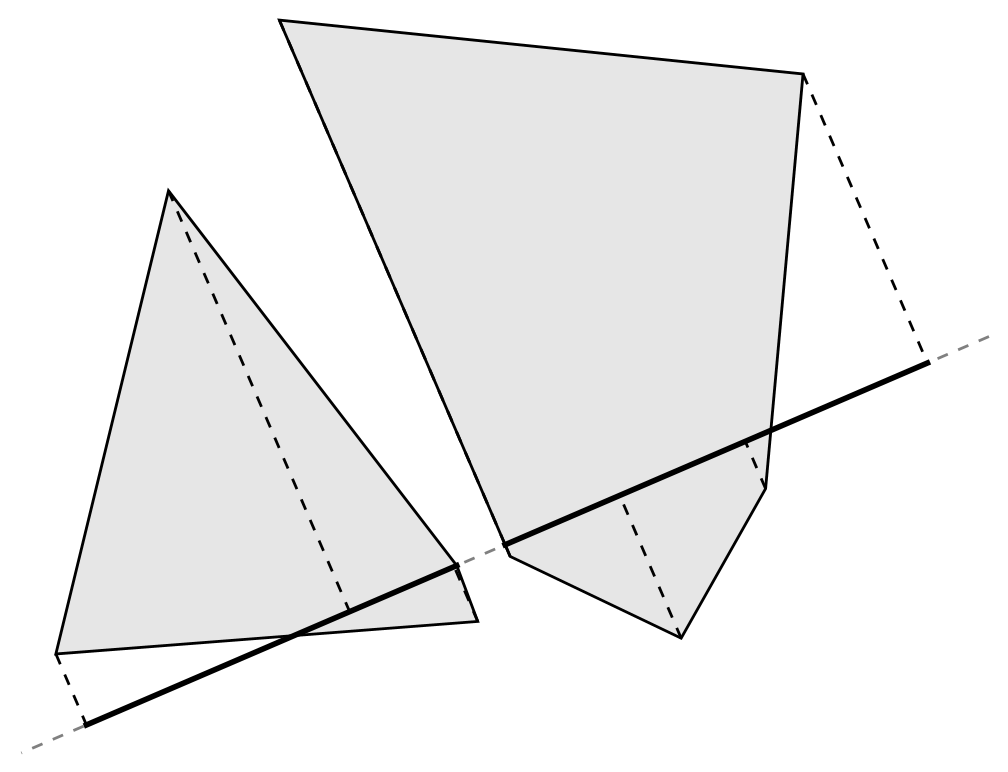}		
        \caption{\label{fig:Proof}
        The left panel shows that if there is a separating axis, the dashed gray line then convex sets are disjoint and the right panel shows that if two convex sets are disjoint then the axis connecting their closest points separates.}
        \label{fig:Separate}
	\end{figure}
\FloatBarrier
\end{proof}

Looking at the proof above, our approach for determining whether two figures do not intersect, requires us to find a separating axis (see Figure \ref{fig:Separate}). This is carried out by looking for the closest pair of points which we do by solving the quadratic program 
$$\min\limits_{p_U\in U,p_V\in V}\|p_U-p_V\|^2$$
the procedure is described in Algorithm~\ref{alg:ConvexIntersect}.

\begin{algorithm}
\caption{Determine whether convex polytopes intersect}\label{alg:ConvexIntersect}
\begin{algorithmic}
\Procedure{ConvexIntersection}{ConvexPolytope U,ConvexPolytope V}
\If{axis aligned bounding boxes of $U$ and $V$ do not intersect}
\State\Return False 
\EndIf
\State $\{p_U,p_V\}\gets\min\limits_{p_U\in U,p_V\in V}\|p_U-p_V\|^2$
\State$\vec{d}\gets p_U-p_V$
\If{$\pi_{\vec{d}(U)}\cap\pi_{\vec{d}(V)}=\varnothing$}
\State\Return False
\Else
\State\Return True
\EndIf
%\State\Return$\pi_{\vec{d}(U)}\cap\pi_{\vec{d}(V)}\neq\varnothing$
\EndProcedure
\end{algorithmic}
\end{algorithm}

Using these ideas from computational geometry allows us to automate determining sparsity patterns for operators. This means that we do not have to compute sparsity patterns during mesh generation. This assists in the implementation of methods such as multigrid where coarse and fine meshes may be generated separately, and therefore the sparsity of the interpolation and prolongation operators cannot be computed at mesh generation. These approaches also enable easier exploration of novel mesh generation techniques and grants the flexibility to mix and match different tools for different parts of the problem.

\FloatBarrier\section{Quadrature}\label{sec5}
In order to implement the operators $a$ and $\ell$, we will need quadrature methods. Since the supports of our basis functions are unions of simplices, we implement quadrature over simplices. In particular, we use the $n$-dimensional degree $d$ rules of Grundemann and M\"{o}ller.
\begin{theorem}\cite[Theorem 4]{GrundmannMoller}
    Let $n \in \mathbb{N}$ be the dimension of the simplex being integrated over, $S_n$ be the standard simplex of dimension $n$, $s\in\{0,1,2,3,\cdots\}$, and $d=2 s+1$ be the intended degree of quadrature. Then the quadrature method
    \begin{align*}
        Q(p)=\sum_{i=0}^s(-1)^i 2^{-2s}\frac{(d+n-2i)^d}{i!(d+n-i)!}&\sum_{|\mathbf{\beta}|=s-i,\beta_0\geq\cdots\geq\beta_n}p\left(\left(\frac{2\beta_0+1}{d+n-2i},\cdots,\frac{2\beta_n+1}{d+n-2i}\right)\right)\\=&\int_{S_n}p(\mathbf{x})d\mathbf{x}
    \end{align*}
is exact for polynomials of degree $d$ (cf. \eqref{permsum} for the bracket notation).
\end{theorem}
The degree $d$, dimension $n$ Grundmann and M\"{o}ller method requires $\binom{\frac{2n+d+1}{2}}{\frac{d-1}{2}}$ function evaluations. This is exponential in the degree, which is usually kept small, and polynomial in the dimension, which may not be small. Compare this to tensor product methods whose costs are $(\frac{d+1}{2})^n$ function evaluations which is exponential in the dimension and polynomial in the degree. This quickly leads to tensor product methods being infeasible. This is summarized in Table \ref{tab:QuadCosts}

\begin{table}
    \centering
    \begin{tabular}{|c|c|c|c|} \hline 
         &  $d=3$&  $d=5$&  $d=7$\\ \hline 
         $n=1$&  $3$&  $6$&  $10$\\ \hline 
         $n=2$&  $4$&  $10$&  $20$\\ \hline 
         $n=3$&  $5$&  $15$&  $35$\\ \hline 
         $n=4$&  $6$&  $21$&  $56$\\ \hline 
         $n=5$&  $7$&  $28$&  $84$\\ \hline 
 $n=6$& $8$& $36$&$120$\\ \hline 
 $n=7$& $9$& $45$&$165$\\ \hline 
 $n=8$& $10$& $55$&$220$\\ \hline 
 $n=9$& $11$& $66$&$286$\\ \hline 
 $n=10$& $12$& $78$&$364$\\\hline \hline
 & $n+2$& $\frac{n^2+5n+6}{2}$&$\frac{n^3+9n^2+26n+24}{6}$\\\hline
    \end{tabular}
        \begin{tabular}{|c|c|c|c|} \hline 
         &  $d=3$&  $d=5$&  $d=7$\\ \hline 
         $n=1$&  $2$&  $3$&  $4$\\ \hline 
         $n=2$&  $4$&  $9$&  $16$\\ \hline 
         $n=3$&  $8$&  $27$&  $64$\\ \hline 
         $n=4$&  $16$&  $81$&  $256$\\ \hline 
         $n=5$&  $32$&  $243$&  $1024$\\ \hline 
 $n=6$& $64$& $729$&$4096$\\ \hline 
 $n=7$& $128$& $2187$&$16384$\\ \hline 
 $n=8$& $256$& $6561$&$65536$\\ \hline 
 $n=9$& $512$& $19683$&$262144$\\ \hline 
 $n=10$& $1024$& $59049$&$1048576$\\\hline \hline
 & $2^n$& $3^n$&$4^n$\\\hline
    \end{tabular}
    \caption{The left panel shows the number of function evaluations required for the Grundmann-M\"{o}ller quadrature method of degree $d$ in $n$ dimensions and the right panel shows the number of function evaluations for tensor product quadrature methods of degree $d$ in $n$ dimensions. The bottom row of each table shows the cost of each method for a fixed degree and a general dimension. We can see that tensor product requires far more function evaluations, rapidly becoming infeasible even for modest dimensions where than Grundmann-M\"{o}ller remains feasible even for high dimension problems.}
    \label{tab:QuadCosts}
\end{table}
Since the Grundmann and M\"{o}ller methods do not all have positive weights we need to be more aware of our quadrature error. In order to analyze the error of our quadrature methods, we state the following well-known result.
\begin{theorem}[Bramble-Hilbert \cite{BrennerScott}]
    Let $\Omega$ be convex and let $u\in W_p^m(\Omega)$. Then there exists a polynomial $u^*$ of degree $m$ such that
    $$\|u-u^*\|_{W_p^k(\Omega)}\leq C_{m,n,\gamma} d^{m-k} \|u\|_{W_p^m(\Omega)}$$
    where $p\geq 1$, $k$ is between $0$ and $m$, $d$ is the diameter of the smallest ball containing $\Omega$, $\gamma=\frac{d}{\rho_{max}}$  where $\rho_{max}$ is the radius of the largest ball contained in $\Omega$, and $n$ is the dimension of the space.
\end{theorem}
In particular, taking $p=\infty$ and $k=0$, we have the bound
    $$\|u-u^*\|_{L_\infty(\Omega)}\leq C_{m,n,\gamma} d^{m} \|u\|_{W_\infty^m(\Omega)}$$
and so we have the error-bound.

%%Statement

\begin{theorem}
 In the setting of Bramble-Hilbert, given a quadrature method $Q(f)$ which on the standard simplex $S_n$ is $\sum_k w_k f(x_k)$, and whose value on any other simplex is found by transferring to the standard simplex, we have 
 $$\left|Q(u)-\int_\Omega u \right|\leq \left(\frac{W}{\mu(S_n)}+1\right)\mu(\Omega)C_{m,n,\gamma}d^m \|u\|_{W_\infty^m(\Omega)}$$
where $W=\sum_k |w_k|$.
\end{theorem}
\begin{proof}
From Bramble-Hilbert we have that 
$$u=u^*+R$$
where $u^*$ is a polynomial of degree $m$ and $\|R\|_{L_\infty(\Omega)}\leq C_{m,n,\gamma}d^m \|u\|_{W_\infty^m(\Omega)}$
so,
\begin{align*}
    \left|Q(u)-\int_\Omega u \right|&=\left|Q(u^*+R)-\int_\Omega u^*+R \right|=\left|Q(R)-\int_\Omega R\right|\\&\leq\left|Q(R)\right|+\left|\int_\Omega R\right|\leq\left(W\frac{\mu(\Omega)}{\mu(S_n)}+\mu(\Omega)\right)\|R\|_{L_\infty(\Omega)}\\&\leq\left(n!W+1\right)\mu(\Omega)C_{m,n,\gamma}d^m \|u\|_{W_\infty^m(\Omega)} \, .
\end{align*}
The proof is complete.
\end{proof}

In the context of this theorem we can analyze the Grundmann and M\"{o}ller method. For methods with positive weights like tensor product methods, the Lebesgue Constant $\left(n!W+1\right)$ is always equal to $2$ where for Grundmann and M\"{o}ller methods the constant is \cite[Theorem 4.7]{GrundmannMoller}
$$1+n!\sum_{i=0}^\frac{d-1}{2}\frac{2^{1-d}(d-2i+n)^d}{i!(d-i+n)!}\binom{n+\frac{d-1}{2}-i}{n}\in \frac{1}{(\frac{d-1}{2})!}\left(\frac{n}{2}\right)^\frac{d-1}{2}+O\left(n^\frac{d-3}{2}\right)$$
This is summarized in table \ref{tab:ErrorFactor} where we can see that for reasonable degrees and relatively high dimensions we lose less than $2$ digits to quadrature error. 
\begin{table}
    \centering
    \begin{tabular}{|c|c|c|c|} \hline 
         &  $d=3$&  $d=5$&  $d=7$\\ \hline 
         $n=1$&  $2.67$&  $4.13$&  $7.21$\\ \hline 
         $n=2$&  $3.13$&  $5.26$&  $9.69$\\ \hline 
         $n=3$&  $3.60$&  $6.63$&  $12.94$\\ \hline 
         $n=4$&  $4.08$&  $8.25$&  $17.07$\\ \hline 
         $n=5$&  $4.57$&  $10.13$&  $22.21$\\\hline
 $n=6$& $5.06$& $12.25$&$28.48$\\\hline
 $n=7$& $5.56$& $14.63$&$36.01$\\\hline
 $n=8$& $6.05$& $17.25$&$44.91$\\\hline
 $n=9$& $6.55$& $20.13$&$55.31$\\\hline
 $n=10$& $7.04$& $23.25$&$67.34$\\\hline
 & $\frac{n}{2}+O(1)$& $\frac{n^2}{8}+O(n)$&$\frac{n^3}{48}+O(n^{2})$\\\hline
    \end{tabular}
    \caption{This table shows the error factor for Grundmann-M\"{o}ller quadrature. Let $\mathrm{ErrorFactor} = n! W + 1$, then     
    $\log_{10}(\frac{\mathrm{Error Factor}}{2})$ is the number of digits lost by using Grundmann-M\"{o}ller quadrature of degree $d$ in $n$ dimensions instead of tensor product quadrature. For instance, for $n=10$ and $d=7$ you lose $\log_{10}\frac{67.34}{2}\approx 1.53$ digits and have to go out to dimension $42$ in order to lose a third digit. For $n=10$ and $d=5$ you lose $\log_{10}\frac{23.25}{2}\approx 1.06$ digits and have to go out to dimension $124$ to lose three digits of accuracy. The bottom row of the table shows the growth rate of the error factor for fixed degree and general dimension.}
    \label{tab:ErrorFactor}
\end{table}

So, for integrands that are exactly polynomial we get an exact quadrature method and for other integrands we can control the error by either the diameter of the element or the degree of the quadrature method. By being careful about our quadrature errors, we find that we can solve high dimensional quadrature problems with well understood accuracy and small numbers of points which allows us to solve problems in much higher dimensions than we could hope to with tensor product methods. This is important to problems with uncertainty, see for instance \cite{GJLord_CEPowell_TShardlow_2014a,HAntil_SDolgov_AOnwunta_2022b}. Being clear about which quadrature methods we are using increases the transparency of the library and produces a clearer academic record.

\FloatBarrier\section{Numerical Experiments} \label{sec6}
In order to demonstrate the usage and general applicability of the tool we will show four examples: a simple Poisson problem, a convection-diffusion problem with impulse data, a non-linear diffusion problem, and Stokes equation for colliding flow. But, before we get to the examples, we need to discuss in general how to solve a variational problem using this tool. Solving a variational problem using this tool can be broken into the following steps:
%\subsubsection*{Define mesh geometry}
\begin{enumerate}
	\item[A.] Define mesh geometry: To define the mesh geometry you need to provide a triangulation of your domain in the format of a structure array with fields:
    \begin{enumerate}    
	\item[a.] $\mathtt{Points}:$ The vertices of the triangulation.
	\item[b.] $\mathtt{ConnectivityList}:$ The list of $n$-dimensional simplices.
	\item[c.] $\mathtt{BoundaryFacets}:$ The list of $n-1$-dimensional facets that make up the boundary of the domain.
	\item[d.] $\mathtt{ActiveBoundary}:$ The list of $n-1$-dimensional facets that make up the Neumann boundary of the domain.
    \end{enumerate}
	\item[B.] Define finite element space: To define a finite element space you need to decide what your trial and test spaces are and what kinds of finite elements you are going to use to form a basis for your discretized trial and test spaces. 
	\item[C.] Define operators: In order to define the appropriate operators, you need to define the binary form, which is a function which takes one finite element basis function and another finite element basis function and returns a scalar, and the linear operator which takes a finite element basis function and returns a scalar.
	
	\item[D.] Prepare and solve system: Once you have the operators you need to either set up the matrix-vector problem $Ax=b$ for linear systems or the root finding problem $F(x)=b$ for non-linear systems. That is, you need to as appropriate provide the vector $b$, and either the matrix $A$, or the function $F$ and its Jacobian. Once you have these you can, for linear systems, use MATLAB's built in linear system solvers, and for non-linear problems, we provide an implementation of Newton's method. 
\end{enumerate}
\FloatBarrier\subsection{Poisson Problem}\label{ex1}
	Consider the standard Poisson problem:
	\begin{align} 
	-\Delta u&=f \textrm{ in } \Omega=(-1,1)\times(-1,1), \label{prob1}\\
	u&=0\textrm{ on }\partial\Omega, \label{prob2}
	\end{align}
	where $f=6(x^2+y^2-2)xy $ and $u=-(x^3-x)(y^3-y)$. We will solve this using a uniform triangular mesh \texttt{TR}. The finite element space is defined in such a way that it enforces the Dirichlet boundary conditions. We take a basis $\{\phi_i\}$, made up of $\mathbf{P}_1$ basis functions, for
	\begin{align*}
	H^1_0:=\{u\in H^1(\Omega)~;~u\textrm{ is continuous, } u=0\textrm{ on }\partial\Omega\textrm{, and } \\ \hspace*{1cm} u\textrm{ is linear on faces of the triangulation}\}
	\end{align*}
	This basis can be used for both the trial and test spaces. We can construct this basis using \texttt{SpaceUtils.Mesh2Basis}.
	The variational form for \eqref{prob1}--\eqref{prob2} is: find $u\in H_0^1$ such that 
	$$\int_\Omega \nabla u\cdot\nabla v=\int_\Omega fv$$
	for all $v\in H_0^1$. We have the bilinear operator
	$$a(u,v)=
	\int_\Omega \nabla u\cdot\nabla v$$
	which we implement as \texttt{FastBinaryOperators.Gradient}  and the linear operator
	$$L(v)=\int_\Omega f v$$ 
	which we implement as \texttt{FastLinearOperators.L2}. Finally, we construct the matrix-vector problem ${\bm A} {\bm u}={\bm b}$ where $A_{ij}=a(\phi_i,\phi_j)$ and $b_i=L(\phi_i)$ and solve it for $\bm u$ using MATLAB's backslash operator.
	\begin{figure}[hbt!]		
		\centering
		\includegraphics[width=0.45\textwidth,trim={0 0 0 1.2cm},clip]{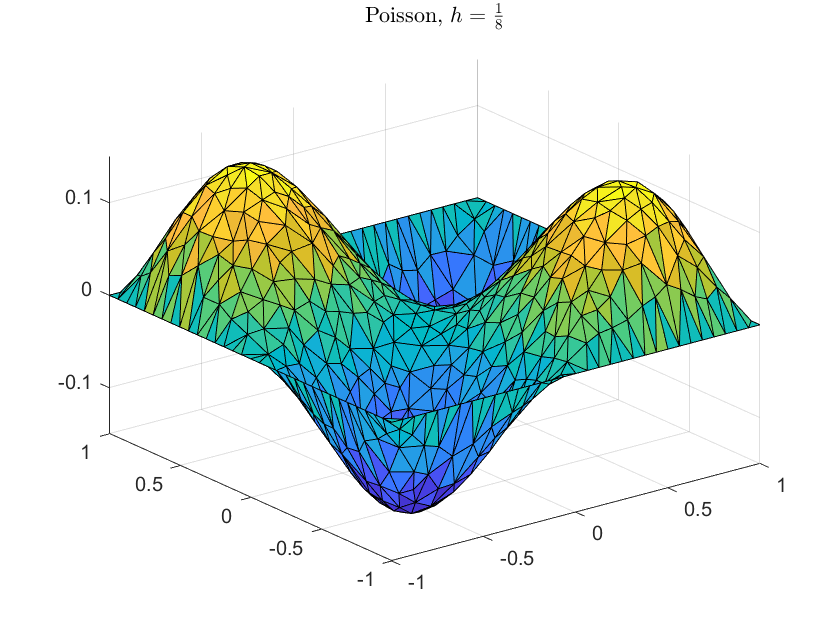}
        \includegraphics[width=0.45\textwidth,trim={0 0 0 1.6cm},clip]{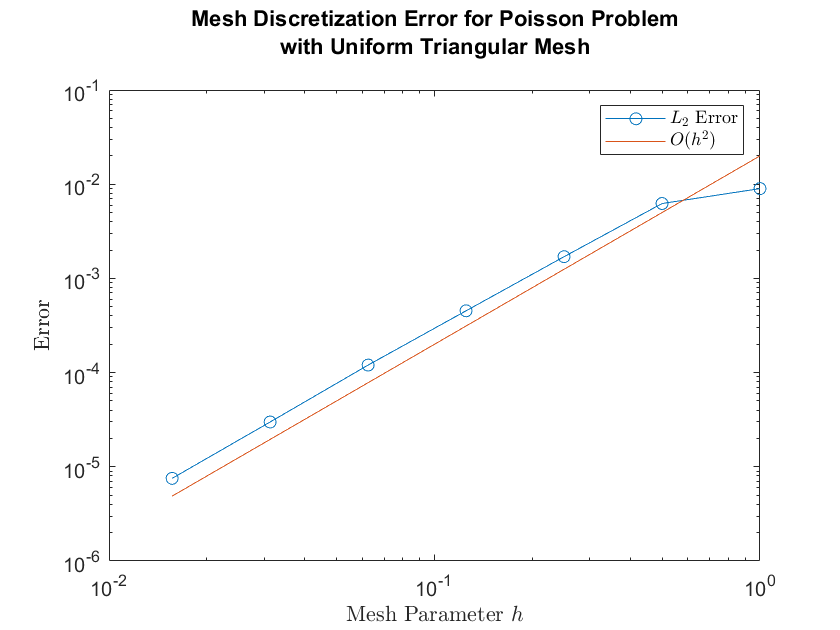}
        \caption{Left panel shows the approximate solution of  Example \ref{ex1} using
  a uniform triangular mesh with maximum side length $h= \frac{1}{8}$.
        The right panel shows $L_2$ error in approximating Example \ref{ex1} using uniform triangular meshes, varying the maximal side length $h$. The expected quadratic rate of convergence is observed. 
        }
		\label{fig:Poisssol}        
	\end{figure}
	%%Add solution image
	
	When solving the Poisson problem this way we expect to see the error estimate
	$$\|u-u_h\|_{L_{2}(\Omega)} \leq C h^2,$$
	where $u_h$ is the discrete solution and $h$ is the longest side length of the triangles in the mesh.
	%%(8.1.12 BS) 
    Figure~\ref{fig:Poisssol} shows the solution and convergence error plots.
    %\todo{The legend in the figure should say $u$ instead of $p$. Also, please use $L_2$ instead of 2 for norm.}
	%The discrete solution $u_h$ is plotted and it is shown in Figure \ref{fig:Poisssol}. Figure \ref{fig:Poiserr} illustrates that the error converges with the optimal rate of $O(h^2)$. 
	% \begin{figure}[hbt!]
	% 	\caption{This figure shows the $L_\infty$ error in approximating Example \ref{ex1} using uniform triangular meshes, varying the maximal side length $h$. We expect our error to scale as $h^2$ and, for an appropriate range of parameters, it does.}
	% 	\centering
	% 	\includegraphics[width=\textwidth,trim={0 0 0 1.6cm},clip]{PoissonError.png}
	% 	\label{fig:Poiserr}
	% \end{figure}

\FloatBarrier\subsection{Convection-Diffusion Problem}\label{ex2}
	In this example, we consider the following problem:
	\begin{align}
	-\varepsilon\Delta u-u_y&=M\thickspace
 \delta_p(x) \textrm{ in } \Omega=(0,1)\times(0,2),\label{prob3}\\
	u&=0\textrm{ on }\Gamma_D:=\{(x,y)\, |\, x=0\textrm{, }x=1\textrm{, or }y=2\},\label{prob4}\\
	u_y&=0\textrm{ on }\Gamma_N:=\{(x,y)\,|\, y=0\},\label{prob5}
	\end{align}
	where $\delta_p$ is the Dirac delta function representing a point source. For our demonstration we have set $M = 1$ and $\varepsilon =1$. We will again solve this using a uniform triangular mesh \texttt{TR}. Next, we define our finite element space again to enforce the Dirichlet boundary conditions. We choose a basis $\{\phi_i\}$ of $\mathbf{P}_1$ functions for
    \begin{align*}
	H_{E_0}^1:=\{u\in H^1(\Omega)~;~u\textrm{ is continuous, } u=0\textrm{ on }\Gamma_D\textrm{, and } \\ \hspace*{1cm} u\textrm{ is linear on faces of the triangulation}\}
	\end{align*}
    This basis is again used for both the trial and test spaces.
	The variational form for \eqref{prob3}--\eqref{prob5} is: find $u\in H_{E_0}^1$ such that 
	$$\int_\Omega (\varepsilon \nabla u\cdot\nabla v-u_yv)=M\thickspace v(p).$$
	for all $v\in H_{E_0}^1$. We have the bilinear operator
	$$a(u,v)=\int_\Omega\varepsilon \nabla u\cdot\nabla v-u_yv$$
	which we implement using
	\texttt{FastBinaryOperators.Generic}
	and the linear operator
	$$L(v)=M\thickspace v(p).$$
	Finally, we construct the matrix-vector problem ${\bm A} {\bm u}=\bm b$ where $A_{ij}=a(\phi_i,\phi_j)$ and $b_i=L(\phi_i)$ and solve it for $\bm u$ using MATLAB's backslash operator. The solution of the discrete problem is shown in Figure \ref{fig:CDSolution}. Notice that the goal here is illustrate that this nonsmooth problem can be solved using our approach. 
 %Because this system is non-smooth, we have neither an exact solution nor an error bound.
	\begin{figure}[hbt!]		
		\centering
		\includegraphics[width=0.45\textwidth,trim={0 0 0 1.2cm},clip]{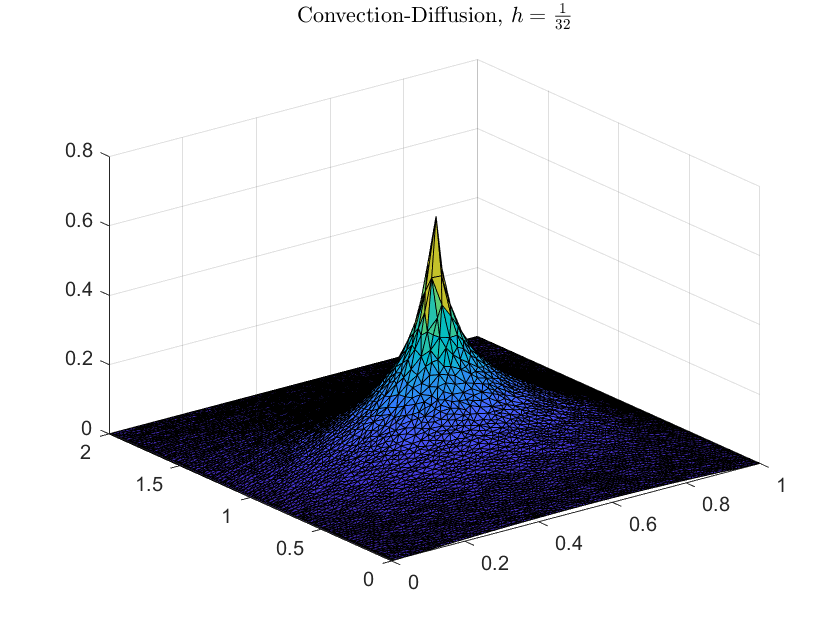}		
  \caption{\label{fig:CDSolution}This figure shows the approximate solution of the convection-diffusion system \ref{ex2} with a unit impulse load located at $(x,y)=(\frac{1}{2},1)$. This is approximated with a uniform triangular mesh whose maximal side length is $h=\frac{1}{32}$.}
	\end{figure}

%%images and discussion
%\FloatBarrier
\subsection{Non-Linear-Diffusion Problem}\label{ex3}
	The framework can also handle non-linear PDEs such as the following:
	\begin{align}
	-u''+u^3&=f:=\pi^2\sin(\pi x)+\sin(\pi x)^3\textrm{ in } \Omega=(0,1),\label{prob6}\\
	u(0)&=u(1)=0,\label{prob7}\end{align}
	which has an exact solution given by 
	$u=\sin(\pi x).$
	We will again solve this using a uniform mesh. 
	We construct our finite element basis $\{\phi_i\}$ of $\mathbf{P}_1$ basis functions satisfying the Dirichlet boundary conditions. This basis is again used for both the trial and test spaces.
	The variational form for \eqref{prob6}--\eqref{prob7} is:
	Find $u\in H_0^1$ such that 
	$$\int_\Omega u' v'+u^3 v=\int_\Omega fv$$
	for all $v\in H_0^1$. We have the binary operator and the linear operator 
	$$a(u,v)=\int_\Omega u' v'+u^3 v \quad \mbox{and} \quad
    L(v) = \int_\Omega fv.$$ 
	%whose Jacobian can be approximated by the Laplacian matrix
	%and the linear operator
	%$$\int_\Omega fv.$$
	Finally, we solve the problem using Newton's method with our approximate Jacobian.
 	\begin{figure}[h!]		
		\centering
		\includegraphics[width=0.45\textwidth,trim={0 0 0 1.2cm},clip]{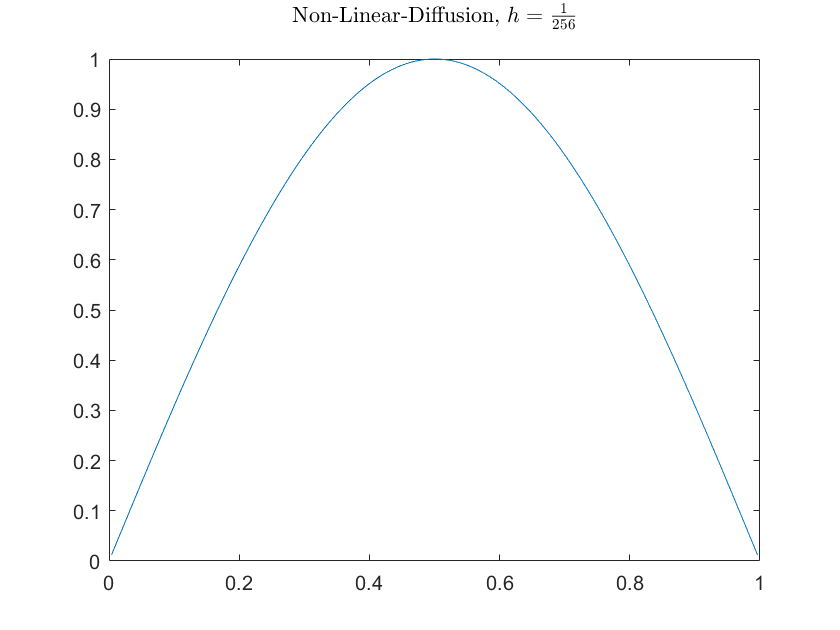}
        \includegraphics[width=0.45\textwidth,trim={0 0 0 1.6cm},clip]{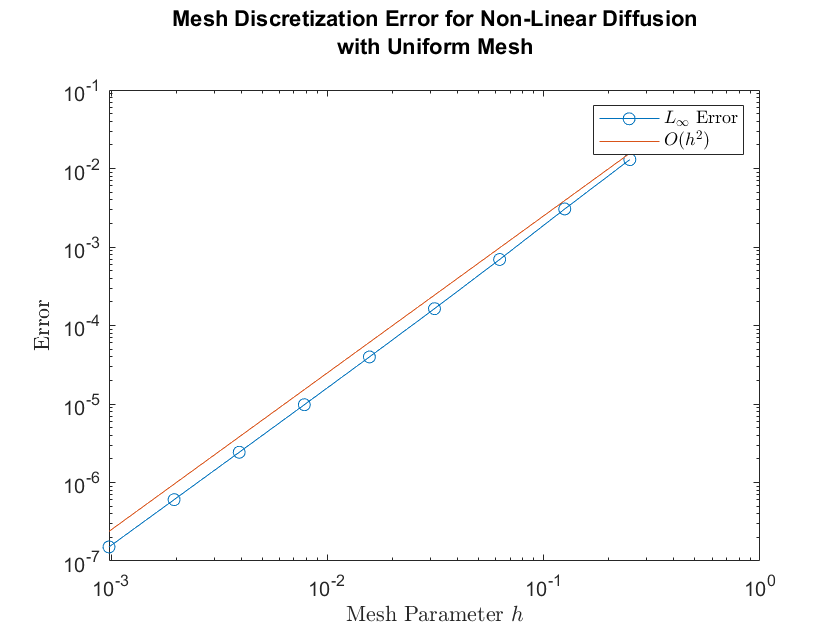}		
        \caption{\label{fig:NLSolution}
        The left panel shows the approximate solution of  Example \ref{ex3} using uniformly spaced nodes with spacing $h= \frac{1}{256}$
        and the right panel shows the quadratic rate of convergence.}
	\end{figure}
 When solving a non-linear-diffusion problem this way we expect to see the error estimate
	$$\|u-u_h\|_{L_\infty(\Omega)} \leq C h^2$$
 where $u_h$ is the discrete solution and $h$ is the spacing between nodes in the mesh. The discrete solution $u_h$ and the error convergence are shown in Figure \ref{fig:NLSolution}. 
 %Figure \ref{fig:NLError} shows that the error converges with the expected rate $O(h^2)$.
	% \begin{figure}[hbt!]
	% 	\caption{This figure shows the $L_\infty$ error in approximating Example \ref{ex3} using uniformly spaced nodes, varying the node spacing $h$. We expect our error to scale as $h^2$ and, for an appropriate range of parameters, it does.}
	% 	\centering
	% 	\includegraphics[width=\textwidth,trim={0 0 0 1.6cm},clip]{NonLinearError.png}
	% 	\label{fig:NLError}
	% \end{figure}
%\FloatBarrier
\subsection{Stokes Problem} \label{ex4}
	Next, we consider the Stokes equation for colliding flow:
	\begin{align*}
	-\Delta \Vec{u}+\nabla p&=0 \textrm{ in } \Omega=(-1,1)\times(-1,1)\\
	\nabla\cdot \Vec{u}&=0 \textrm{ in } \Omega\\   
	\Vec{u}&=\binom{20xy^3}{5x^4-5y^4}\textrm{ on }\partial\Omega
	\end{align*}
	whose solution for pressure is
	$$p=60x^2y+20y^3+C.$$
	We will solve this using a uniform triangular mesh \texttt{TR}. 
 %The velocity finite element space imposes the Dirichlet boundary conditions. 
	%
	For pressure, we construct a basis $\{\psi_i\}$, made up of $\mathbf{P}_1$ basis functions, for
	$$\mathbf{L}_2:=\{u\in L_2(\Omega)\, |\, u\textrm{ is continuous, and } u\textrm{ is linear on faces of the triangulation}\}$$
	This basis will be used as both the trial and test space for the pressure component. 
    For velocity, we construct a basis $\{\phi_i\}$, made up of $\mathbf{P}_2$ basis functions
	$$\mathbf{H}^1_{E_0}:=\{u\in H^1(\Omega)\, |\, u\textrm{ is continuous, } u=\Vec{0}\textrm{ on }\partial\Omega\textrm{, and } u\textrm{ is quadratic on faces of the  triangulation}\}$$
	and we use $\{(\phi_1,0),\cdots,(\phi_n,0),(0,\phi_1),\cdots,(0,\phi_n)\}\in\vec{\mathbf{H}}^1_{E_0}$ as the basis for both the trial and test space for the velocity component.
	
The variational form for this problem is: find $\vec{u}\in\vec{\mathbf{H}}^1_{E_0}$ and $p\in\mathbf{L}_2$ such that
 \begin{align*}
     \int_\Omega\nabla \vec{u}:\nabla \vec{v}-\int_\Omega p\nabla\cdot\vec{v}&=0\textrm{ for all }\vec{v}\in\vec{\mathbf{H}}^1_{E_0},\\
     \int_\Omega q\nabla\cdot\vec{u}&=0\textrm{ for all }q\in\mathbf{L}_2.
 \end{align*}
 
This, along with our velocity and pressure bases, leads to the matrix vector equation 
$$\begin{bmatrix}\bm A&\bm B^\top\\ \bm B&\bm 0\end{bmatrix}\begin{bmatrix}\bm u\\ \bm p\end{bmatrix}=\begin{bmatrix}\bm f\\ \bm g\end{bmatrix}.$$
%$complicated math$ where $MATH$. 
This system is solved using MATLAB backslash operator. Some results are shown in 
Figure~\ref{fig:StokesSolution}.

%%%Add analysis
	\begin{figure}[hbt!]		
		\centering
		\includegraphics[width=0.45\textwidth,trim={0 0 0 1.6cm},clip]{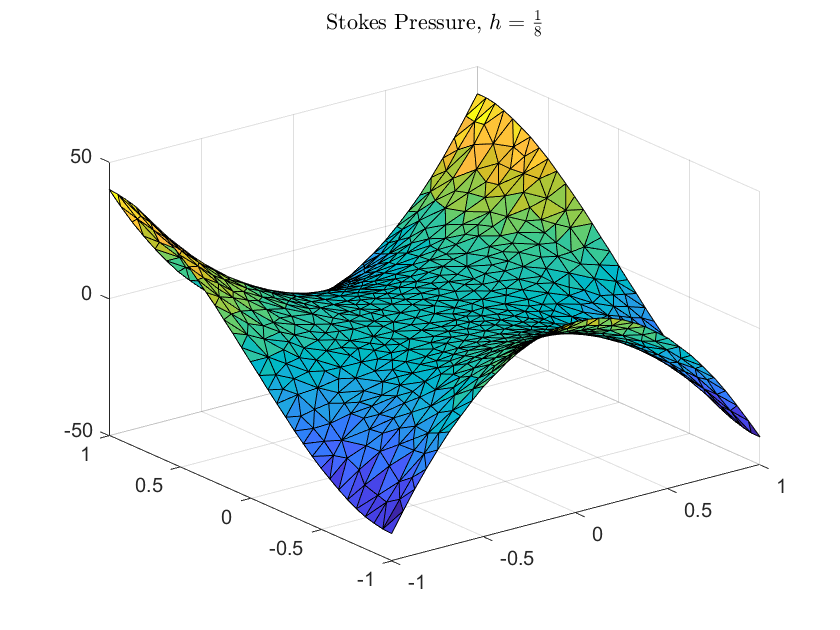}
        \includegraphics[width=0.45\textwidth,trim={0 0 0 1.6cm},clip]{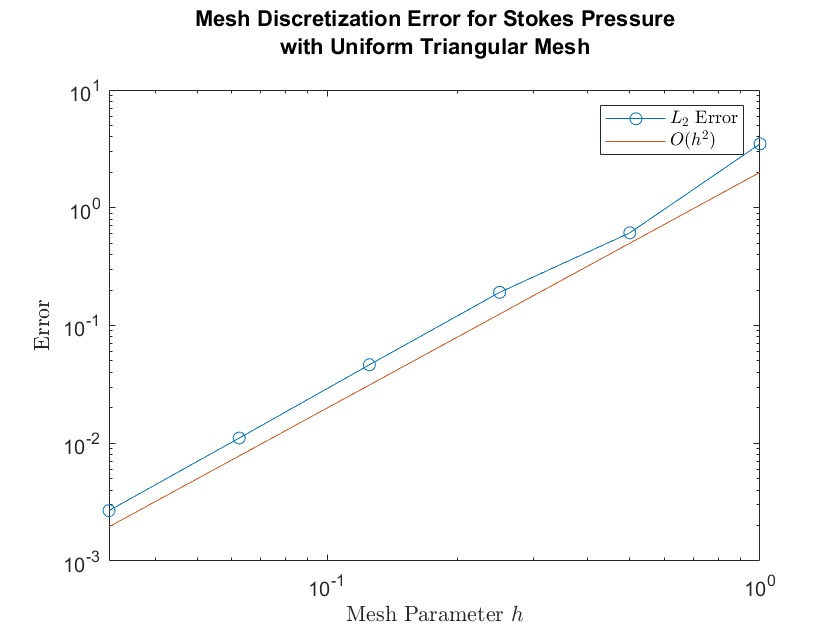}		
        \caption{\label{fig:StokesSolution}The left panel shows the approximate solution for pressure of Example \ref{ex4} using a uniform triangular mesh with maximum side length $h= \frac{1}{8}$. The right panel shows the quadratic rate of convergence for pressure.}
	\end{figure}
	%%Add solution image
	
	% When solving the Stokes problem this way we expect to see the error estimate
	% $$MATH$$
	% where $MATH$ 
	
	% Note that the figure below shows that we obtain this result. 
	% \begin{figure}[hbt!]
	% 	\caption{This figure shows the $L_2$ error in approximating Example \ref{ex4} using uniform triangular meshes, varying the maximal side length $h$. We expect our error to scale as $h^2$ which it does.}
	% 	\centering
	% 	\includegraphics[width=\textwidth,trim={0 0 0 1.6cm},clip]{StokesError.png}
	% 	\label{fig:StokesError}
	% \end{figure}
%%images and discussion
\FloatBarrier\printbibliography

\end{document}